\documentclass[a4paper]{elsarticle}
\usepackage{amssymb,amscd,amsthm,latexsym,qsymbols, mathrsfs }
\usepackage[all]{xy}
\usepackage[british]{babel}
\usepackage{amsmath}
\usepackage{amsfonts}
\usepackage{amssymb}
\usepackage{amsthm}
\usepackage[all]{xy}

\usepackage{textcomp}

\theoremstyle{plain}
\newtheorem{theorem}{Theorem}[section]
\newtheorem{lemma}[theorem]{Lemma}
\newtheorem{obs}[theorem]{Observation}
\newtheorem{proposition}[theorem]{Proposition}
\newtheorem{corollary}[theorem]{Corollary}

\theoremstyle{definition}

\newtheorem{definition}[theorem]{Definition}

\theoremstyle{remark}

\newtheorem{example}[theorem]{Example}

\newtheorem{remark}[theorem]{Remark}

\begin{document}

\title{Weighted commutators in semi-abelian categories}

\author[MG]{Marino Gran\fnref{fn1,fn3}}
\ead{marino.gran@uclouvain.be}

\author[GJ]{George Janelidze\fnref{fn2,fn3}}
\ead{george.janelidze@uct.ac.za}

\author[AU]{Aldo Ursini\fnref{fn3}}
\ead{ursini@unisi.it}

\fntext[fn1]{Supported by  F.N.R.S. grant \emph{Cr\'edit aux chercheurs} 1.5.016.10F}
\fntext[fn2]{Partially supported by South African National Research Foundation}
\fntext[fn3]{The authors gratefully acknowledge kind hospitality of Universit\'e catholique de Louvain and of University of Cape Town}

\address[MG]{Institut de Recherche en Math\'ematique et Physique, Universit\'e catholique de Louvain, Chemin du Cyclotron 2, 1348 Louvain-la-Neuve, Belgium}
\address[GJ]{Department of Mathematics and Applied Mathematics, University of Cape Town, Rondebosch 7701, Cape Town, South Africa}
\address[AU]{Dipartimento di Ingegneria dell'Informazione e Scienze Matematiche, Universit\'a di Siena, via Roma 56, 53100 Siena, Italy, and \\ Mathematics Division, Department of Mathematical Sciences, Stellenbosch University, Private Bag X1 Matieland 7602, South Africa}

\begin{abstract}
We introduce new notions of Òweighted centralityÓ and Òweighted commutatorsÓ
corresponding to each other in the same way as centrality of congruences and commutators
do in the Smith commutator theory. Both the Huq commutator of subobjects and
Pedicchio's categorical generalization of Smith commutator are special cases of our
weighted commutators; in fact we obtain them by taking the smallest and the largest
weight respectively. At the end of the paper we briefly consider the universal-algebraic
context in connection with an older work of the third author on the ideal theory version of
the commutator theory.
\\


\noindent AMS Math. Subj. Class.: 18C05,  18E10, 08A30.
\end{abstract}

\maketitle

\section*{Introduction}
Classically, subgroups $X$ and $Y$ of a group A are said to centralize each other if $xy = yx$ for
every $x \in X$ and $y \in Y$. When $X$ and $Y$ are normal subgroups in $A$, so is their commutator
$[X,Y]$, and, moreover, it is the smallest normal subgroup $C$ in $A$, for which
$$(xC)(yC) = (yC)(xC)$$
in $A/C$ for every $x \in X$ and $y \in Y$. These concepts of centrality (or centralization) and
commutator have been generalized many times and in various ways from groups to
various types of universal algebras and even further to objects of abstract categories
satisfying suitable exactness conditions. Let us mention three of these generalizations
relevant for our purposes:
\subsection*{Commuting arrows and Huq commutator} Let $\mathbb C$ be a pointed category with finite
products. Two morphisms $x \colon X \rightarrow A$ and $y \colon Y \rightarrow A$ in $\mathbb C$ with the same codomain are said
to commute, if there exists a morphism $m : X\times Y \rightarrow A$ making the diagram
\begin{equation}\label{Huq}
          \xymatrix@=30pt{X  \ar[r]^-{\langle 1,0\rangle } \ar[dr]_{x} & X \times Y \ar@{.>}[d]^{m} & Y  \ar[l]_-{\langle 0, 1 \rangle } \ar[dl]^{y} \\ & A, &
         }
\end{equation}
commute; here $1$ and $0$ denote suitable identity and zero morphisms respectively. It is
easy to see that in the case of groups and $x \colon X \rightarrow A$ and $y \colon Y \rightarrow A$ being the inclusion
maps, $x$ and $y$ commute in this sense if and only if every element of $X$ commutes with
every element of $Y$ in the usual sense. The notion of commuting morphisms was first
introduced and studied by S. A. Huq \cite{H}  in a context closely related to semi-abelian;
moreover, the so-called old style axioms for semi-abelian categories \cite{JMT} are essentially
Huq's axioms. More recently, the notion of commuting morphisms has been examined in
different categorical contexts by D. Bourn and his collaborators (see \cite{Intrinsic, BG, BB} and
references therein). As suggested by the group case, what is meant by the Huq
commutator of $X \rightarrow A$ and $Y \rightarrow A$ is the smallest normal subobject $C \rightarrow A$ of $A$, for which
the composites
$$X \rightarrow A \rightarrow Coker(C \rightarrow A) \qquad  {\rm and} \qquad Y \rightarrow A \rightarrow Coker(C \rightarrow A)$$
commute. In particular, the Huq commutator always exists in any semi-abelian category,
since it is the kernel of its cokernel that can be presented as the colimit of the diagram of
solid arrows in (1), as observed by D. Bourn \cite{Strongly}.
\subsection*{Congruences centralizing each other, and Smith commutator} Let $\mathbb C$ be a
Mal'tsev (=congruence permutable) variety of universal algebras, and $R$ and $S$ be
congruences on the same algebra $A$ in $\mathbb C$. Reformulating the original definition of J. D. H.
Smith \cite{S} in a slightly different language, $R$ and $S$ are said to centralize each other by
means of the centralizing congruence $E$, if there is an internal double equivalence
relation (=Òdouble congruenceÓ) in $\mathbb C$ of the form

$$\xymatrix@=30pt{ E \ar@<1ex>[d]^{}
\ar@<-1ex>[d]_{}
\ar@<1ex>[r]^{}
\ar@<-1ex>[r]_{}
 & S \ar@<1ex>[d]^{} \ar@<-1ex>[d]_{}
 \\
R \ar@<1ex>[r]^{} \ar@<-1ex>[r]_{}
 & A,
 }
$$
where the bottom row and the right-hand column represent $R$ and $S$ respectively, and the
whole diagram forms two parallel discrete fibrations of equivalence relations. The Smith
commutator $[R,S]_{\mathsf{Smith}}$ can then be defined as the smallest congruence $C$ on $A$, for which the
congruences on $A/C$ induced by $R$ and $S$ centralize each other. This commutator was
introduced and studied in \cite{S}, and then generalized in various ways by other authors,
who also proposed several equivalent definitions. Out of those we will only consider:
\subsection*{Categorical version of Smith commutator due to Pedicchio} M. C. Pedicchio \cite{P,P1},
not only extended Smith's definitions to the context of abstract Mal'tsev categories,
but almost at the same time discovered that the existence of centralizing congruence for $R$
and $S$ in the Smith definition above is equivalent to the existence of the Kock
pregroupoid structure on the span $\xymatrix{A/X & A \ar[l] \ar[r] & A/Y}$. Still equivalently, $R$ and $S$ centralize
each other if and only if there exists a morphism $p : R \times_A S \rightarrow A$ making the diagram
\begin{equation}\label{partial}
  \xymatrix@=40pt{R  \ar[r]^-{\langle r_1, r_2, r_2  \rangle } \ar[dr]_{r_1} & R \times _A S  \ar@{.>}[d]^{p} & S  \ar[l]_-{\langle s_1, s_1, s_2  \rangle } \ar[dl]^{s_2} \\ & A&
 }
\end{equation}
in which $R = (R,r_1,r_2)$, $S = (S,s_1,s_2)$, and $R\times_A S = R \times_{(r_2,s_1)} S$, commute. That is, such a
morphism $p$ exists if and only if $[R,S]_{\mathsf{Smith}}$ is the equality relation $\Delta_A$ on $A$.

 In a finitely complete pointed Mal'tsev category both the notions of centra-lity of morphisms and of equivalence relations can be defined. It was then a natural question to determine the precise relationship between them: this question was first considered in \cite{BG}, where it was proved that the existence of a partial Mal'tsev operation $p \colon R \times_A S \rightarrow A$ always makes the corresponding normal monomorphisms $n_R \colon N_R \rightarrow A$ and $n_S \colon N_S \rightarrow A$ commute.
The converse implication, saying that the equivalence relations $R$ and $S$ centralise each other whenever the corresponding normal monomorphisms commute, turned out to be true both in strongly protomodular categories \cite{BG} and in action accessible categories \cite{BJ}. However, this latter implication fails for a general semi-abelian category \cite{JMT}, as a counter-example found by the second author in the semi-abelian variety of digroups shows (see \cite{Strongly}). Furthermore, Exercise $10$ of Chapter $5$ in \cite{FM},
published much earlier, in fact also gives such a counter-example, but in the
variety of loops; surely it was not noticed before only because the Huq
commutator is not mentioned in \cite{FM}.

% The notion of pregroupoid was first introduced by A. Kock in \cite{Ko}, then studied by M.C. Pedicchio to develop commutator theory in exact Mal'tsev categories \cite{P}, and further extended by G. Janelidze and M.C. Pedicchio to treat commutators in more general contexts %(see for instance \cite{JaP} and references therein).

%In a general semi-abelian category \cite{JMT}, $R$ and $S$ are effective equivalence relations, $(R, r_1, r_2)$ is the kernel pair of the canonical quotient $A \rightarrow A/R$ and $(S, s_1,s_2)$ is the kernel pair of $A \rightarrow A/S$.

In the present article both the Huq centrality of morphisms and the Smith centrality of equivalence relations are shown to be special cases of a new notion, which we call \emph{weighted centrality}. The internal multiplication in the definition of weighted centrality of two morphisms $x \colon X \rightarrow A$ and $y\colon Y \rightarrow A$ will be a morphism $m \colon  (W+X)\times_W (W+Y) \rightarrow  A$ which depends on a ``weight'', which is in fact another morphism $w \colon W \rightarrow A$, as explained in Section \ref{wcentrality}. The Huq centrality is recovered in the case of the zero morphism $0 \colon 0 \rightarrow A$, since $(W+X)\times_W (W+Y)$ then reduces to $X \times Y$ and $m$ to the dotted arrow in diagram $(\ref{Huq})$; the Smith centrality of equivalence relations is recovered in the case where the ``weight'' is the identity morphism $1_A \colon A \rightarrow A$. In the semi-abelian context, this notion of ``weighted centrality'' determines a corresponding notion of \emph{weighted subobject commutator} (Section \ref{subobject}) and of \emph{weighted normal commutator} (Section \ref{weightednormal}), whose study we begin in the present article. We compare these notions with the Huq commutator of morphisms in Section \ref{SectionHuq}, and with the categorical notion of commutator of equivalence relations \cite{P} in Sections \ref{Epimorphic} and \ref{Pedicchio}. In the last section we relate our constructions to certain
universal-algebraic ones involving the commutator of subalgebras.

A  notion of commutator of subalgebras, due to the third author, appears in \cite{UPreprint, GU, Urs94}, and it
corresponds to
the Smith commutator of congruences in any semi-abelian variety $\mathbb C$
(and, more generally, to the so-called modular commutator in ideal determined varieties).
In the case of two normal subalgebras $N_R$ and $N_S$ of an algebra $A$, which are the $0$-classes of two congruences $R$ and $S$, respectively, on $A$, this commutator $[N_R, N_S]_{{\mathbb C}, A}$ has the following remarkable property:
$$[N_R, N_S]_{\mathbb C, A} = 0 \quad \Leftrightarrow \quad [R,S]_{\mathsf{Smith}}= \Delta_A.$$
In the last section we recall this definition of the commutator ${[ \cdot \,, \cdot ]}_{{\mathbb C}, A}$, that is based on a suitable notion of commutator term, and we prove that this is another instance of the normal weighted commutator introduced in Section \ref{weightednormal}.
Accordingly, the present approach provides, in particular, a categorical formulation of this universal-algebraic commutator. We conclude the article  with simple descriptions of the (normal) weighted commutators in the variety of commutative (associative, non-unital) rings.
 \vspace{5mm}

\section{Commuting morphisms and pregroupoids}\label{wcentrality}
Throughout this and the next section we will use the following data: a pointed category $\mathbb C$ with finite limits and colimits, objects $A$, $W$, $X$, $Y$ in $\mathbb C$, and morphisms $w \colon W \rightarrow  A$, $x \colon  X \rightarrow A$, and $
y \colon Y \rightarrow A$. We will eventually assume that $w$ is a monomorphism, and then think of $W = (W,w)$ as a subobject of $A$, and the same applies to $x$ and $y$.

The  \emph{weighted centrality} of two morphisms $x \colon X \rightarrow A$ and $y \colon Y \rightarrow A$ over $w \colon W \rightarrow A$ is defined by the existence of an internal multiplication $X\times Y \rightarrow A$ over $W = (W,w)$ in the following sense:
\begin{definition}\label{defmultiplication}
Given morphisms $w \colon W \rightarrow A$, $x \colon X \rightarrow A$, and $y \colon Y \rightarrow A$, an \emph{internal multiplication} $X\times Y \rightarrow A$ \emph{over} $W = (W,w)$ is a morphism
         $$ m \colon (W + X)\times_W (W+Y) = (W+X)\times_{([1,0],[1,0])}(W+Y) \rightarrow  A,$$
         making the diagram
         \begin{equation}\label{diagram1}
          \xymatrix@=50pt{W+X \ar[r]^-{\langle 1,\iota_1[1,0] \rangle } \ar[dr]_{[w,x]} & (W+X)\times_W (W+Y) \ar[d]^m & W+Y \ar[l]_-{\langle \iota_1 [1,0], 1\rangle } \ar[dl]^{[w,y]} \\ & A &
         }
         \end{equation}
         commute.
\end{definition}
By analogy with the notion of centrality of morphisms \cite{H} recalled in the Introduction, we will also say that $(X,x)$ \emph{and} $(Y,y)$ \emph{commute} \emph{over} $(W,w)$ when there exists an internal multiplication  $X\times Y \rightarrow A$ over $W = (W,w)$.
In this case we then write
$$[(X,x), (Y,y)]_{(W,w)} = 0.$$

 \begin{proposition}\label{existencegamma}
         The composites
$$\xymatrix{(W+X)\times_W (W+Y)  \ar[r]^-{\pi_1} & W+X \ar[r]^-{[w,x]} & A},$$
\begin{equation}\label{defipullback}
\xymatrix{ & W+X \ar[dr]_{[1,0]} \ar[drr]^{[w,0]} & & \\
         (W+X)\times_W (W+Y)  \ar[dr]_-{\pi_2} \ar[ur]^-{\pi_1} & &W  \ar[r]^(.4)w & A,\\
       & W+ Y \ar[ur]^{[1,0]} \ar[urr]_{[w,0]}& &   }
 \end{equation}
and $$\xymatrix{(W+X)\times_W (W+Y)  \ar[r]^-{\pi_2} & W+Y \ar[r]^-{[w,y]} & A},$$
determine a morphism
$$\gamma = \gamma_w = \langle [w,x] \pi_1,[w,0]\pi_1,[w,y]\pi_2\rangle = \langle [w,x]\pi_1,[w,0]\pi_2,[w,y]\pi_2\rangle  $$
from $(W+X)\times_W (W+Y)$ to $A_3$, where $A_3$ is the pullback $A\times_{A/X} A\times_{A/Y} A = A\times_{\mathsf{coker}(x)} A\times_{\mathsf{coker}(y)} A$.
\end{proposition}
\begin{proof}
  This follows from a straightforward calculation using the equalities \\ $\mathsf{coker}(x)x = 0$ and
$\mathsf{coker}(y)y = 0$:
$$\mathsf{coker}(x)[w,x]\pi_1 = [\mathsf{coker}(x)w, \mathsf{coker}(x)x] \pi_1 = [\mathsf{coker}(x)w, 0] \pi_1 = \mathsf{coker}(x)[w,0] \pi_1,
  $$
and similarly $\mathsf{coker}(y)[w,y] \pi_2 = \mathsf{coker}(y)[w,0]\pi_2$.
\end{proof}
Recall that an \emph{internal pregroupoid} structure \cite{Ko} on the span  $$\xymatrix{A/X & A \ar[l] \ar[r] & A/Y}$$ in a Mal'tsev category is a partial Mal'tsev operation $$p \colon A_3 = A \times_{A/X} A \times_{A/Y}  A \rightarrow A$$ making the following diagram commute:
$$
\label{partial}
  \xymatrix@=40pt{A \times_{A/X} A  \ar[r]^-{\langle \pi_1, \pi_2, \pi_2  \rangle } \ar[dr]_{\pi_1} & A_3  \ar@{.>}[d]^{p} & A \times_{A/Y} A  \ar[l]_-{\langle \pi_1, \pi_1, \pi_2  \rangle } \ar[dl]^{\pi_2} \\ & A. &
 }
$$
This precisely means that \emph{the equivalence relations} $(A \times_{A/X} A, \pi_1, \pi_2)$ \emph{and} $(A \times_{A/Y} A , \pi_1, \pi_2)$ \emph{centralise each other}: we then write
$$[A\times_{A/X} A, A\times_{A/X} A]= \Delta_A. $$

  \begin{theorem}\label{pregroupoid}
  If $p \colon A_3 \rightarrow A$ is an internal pregroupoid structure on the span $\xymatrix{A/X & A \ar[l] \ar[r] & A/Y}$, then $p\gamma_w \colon(W+X)\times_W (W+Y) \rightarrow A$ is an internal multiplication $X\times Y \rightarrow A$ over $W = (W,w)$.
  \end{theorem}
 \begin{proof}
 This is again a straightforward calculation:
\begin{eqnarray}
          p \gamma_w\langle1,\iota_1[1,0]\rangle &=& p\langle [w,x]\pi_1,[w,0]\pi_1,[w,y]\pi_2 \rangle \langle 1,\iota_1[1,0] \rangle \nonumber
 \\
          &=& p\langle [w,x]\pi_1\langle 1,\iota_1[1,0]\rangle,[w,0]\pi_1\langle1,\iota_1[1,0]\rangle,[w,y]\pi_2\langle1,\iota_1[1,0]\rangle \rangle \nonumber  \\
          & = & p\langle [w,x],[w,0],[w,y]\iota_1[1,0]\rangle \nonumber  \\
  &=& p\langle [w,x],[w,0],w[1,0]\rangle  \nonumber  \\  &=& p\langle [w,x],[w,0],[w,0]\rangle \nonumber  \\ &=& [w,x] \nonumber
\end{eqnarray}
and similarly $p\gamma_w\langle \iota_1[1,0],1\rangle = [w,y]$ in the notation of Definition \ref{defmultiplication}.
\end{proof}
Recall that a regular category $\mathbb C$ is \emph{homological} \cite{BB} when it is pointed and the Split Short Five Lemma holds in $\mathbb C$ \cite{Bourn0}. In a homological category an internal multiplication is necessarily unique, when it exists, since the arrows $\langle 1,\iota_1[1,0] \rangle$ and $\langle \iota_1 [1,0], 1\rangle$ in diagram (\ref{diagram1}) above are jointly strongly epimorphic (see \cite{Bourn96}).
\begin{lemma}\label{mkeriszero}
Let $\mathbb C$ be a homological category. If $m \colon  (W+X)\times_W (W+Y) \rightarrow A$ is an internal multiplication $X\times Y \rightarrow A$ over
$W = (W,w)$ and $w$ is a monomorphism, then $m(\ker(\gamma_w)) = 0$.
\end{lemma}
\begin{proof}
Let us write $\ker(\gamma_w)$ as $\langle u,v\rangle \colon K \rightarrow (W+X)\times_W (W+Y)$. The equality $\gamma_w\langle u,v\rangle = 0$ means
$$[w,x]\pi_1\langle u,v\rangle = 0,  \, [w,0]\pi_1\langle u,v\rangle = 0 = [w,0]\pi_2 \langle u,v\rangle,  \,  {\rm and}  \, [w,y]\pi_2\langle u,v\rangle= 0$$
or, equivalently,
$$ [w,x]u = 0,   \, [w,0]u = 0 = [w,0]v,   \, {\rm and}   \, [w,y]v = 0.$$
Moreover, since $w$ is a monomorphism, $[w,0]u = w[1,0]u$ and $[w,0]v = w[1,0]v$, we also obtain
$ [1,0]u = 0 = [1,0]v \colon K \rightarrow W$.

This equality implies $[1,0]u\pi_1 = [1,0]v\pi_2 : K\times K \rightarrow W$, and therefore it allows one to present $\langle u,v\rangle$ as the composite $$\xymatrix{ K \ar[r]^-{\langle 1, 1 \rangle}& K \times K \ar[r]^-{u \times v}& (W+X)\times_W (W+Y).}$$

After that, since the morphisms $\langle 1,0\rangle : K \rightarrow K\times K$  and  $\langle 0,1\rangle : K \rightarrow K\times K$ are jointly epic (see \cite{BB}), in order to prove the desired equality $m\langle u,v\rangle = 0$ it suffices to prove that $m(u\times v)\langle 1,0\rangle = 0 = m(u\times v)\langle 0,1\rangle$, or, equivalently, that
$$ m\langle u,0\rangle = 0 = m \langle 0,v \rangle.$$
We have
$$ m\langle u,0\rangle = m\langle u,\iota_1[1,0]u\rangle = m\langle 1,\iota_1[1,0]\rangle u = [w,x]u = 0$$
and, similarly, $$ m\langle 0,v\rangle = m\langle \iota_1[1,0]v,v\rangle = m\langle \iota_1[1,0],1\rangle v= [w,y]v = 0.$$
\end{proof}
The following well-known property will be useful:
\begin{lemma}\label{Pullbackepi}
Consider a commutative diagram in a regular category $\mathbb C$
$$\xymatrix{C_1 \ar[r]^{c_1} \ar[d]_{f_1} & C \ar[d]_f & C_2 \ar[l]_{c_2} \ar[d]^{f_2}\\
D_1 \ar[r]_{d_1}  & D & D_2 \ar[l]^{d_2}
}
$$
where $f_1$, $f_2$ are regular epimorphisms, and $f$ is a monomorphism. Then the induced morphism $ f_1 \times f_2 \colon C_1 \times_C C_2 \rightarrow  D_1 \times_D D_2$ is a regular epimorphism.
\end{lemma}
\begin{theorem}\label{normalepi} If $\mathbb C$ is homological and the morphisms $x \colon X \rightarrow A$ and $y \colon Y \rightarrow A$ have normal images, then the morphism
$$\gamma_1 = \gamma_{1_A} \colon (A+X)\times_A (A+Y) \rightarrow A_3$$ is a normal epimorphism.
\end{theorem}
\begin{proof}
The morphism $\gamma_1$ can be presented as the composite of the morphism
$$\langle [1,x],[1,0] \rangle \times \langle [1,0],[1,y]\rangle = [\langle1,1\rangle,\langle x,0\rangle ] \times [\langle 1,1\rangle,\langle 0,y\rangle ]$$
from $(A+X)\times_A(A+Y)$ to $(A\times_{A/X}A)\times_A (A\times_{A/Y}A)$ with the canonical isomorphism $
(A\times_{A/X}A)\times_A(A\times_{A/Y} A) \cong A_3$. Therefore, by Lemma \ref{Pullbackepi} and the fact that any regular epimorphism in $\mathbb C$ is a normal epimorphism, it suffices to prove that the morphisms $[\langle 1,1\rangle,\langle x,0\rangle]$ and $[\langle 1,1\rangle,\langle 0,y\rangle]$ are normal epimorphisms. This follows, however, from the fact that the diagrams
$$\xymatrix@=35pt{X' \ar[r]_-{\langle x' ,0 \rangle}& A \times_{A/X} A \ar@<-1 ex>[r]_-{\pi_2} & A \ar@<-1 ex>[l]_-{\langle 1, 1 \rangle }} $$
and
$$\xymatrix@=35pt{Y' \ar[r]_-{\langle 0,y' \rangle}& A \times_{A/Y} A \ar@<-1 ex>[r]_-{\pi_1} & A \ar@<-1 ex>[l]_-{\langle 1, 1 \rangle },} $$
where $x' \colon X' \rightarrow A$ and $y' \colon Y' \rightarrow A$ are the images of $x \colon X \rightarrow A$ and $y \colon Y \rightarrow A$ respectively, are split extensions
in $\mathbb C$.
\end{proof}
\begin{lemma}\label{mimpliespregroupoid}
Under the assumptions of Theorem \ref{normalepi}, if $m$ is an internal multiplication $X\times Y \rightarrow A$ over
$A = (A,1_A)$, and $p : A_3 \rightarrow A$ with $p\gamma_1 = m$, then $p$ is an internal pregroupoid structure on the span $\xymatrix{A/X & A \ar[l] \ar[r] & A/Y}.$
\end{lemma}
\begin{proof}
To prove that $p$ is an internal pregroupoid structure on the span \\ $\xymatrix{A/X & A \ar[l] \ar[r] & A/Y}$ is to prove that $p\langle \pi_1,\pi_2, \pi_2\rangle  = \pi_1$ and $p \langle \pi_1,\pi_1,\pi_2\rangle  = \pi_2$, where $\pi_1$ and $\pi_2$ are the two projections $A\times_{A/X} A \rightarrow  A$ in the first equality, and are the two projections $A\times_{A/Y}A \rightarrow A$ in the second equality. Since the arrow
          $$\langle [1,x],[1,0]\rangle = [\langle1,1\rangle, \langle x,0\rangle ] : A+X \rightarrow  A\times_{A/X}A$$
is a normal epimorphism (as explained in the proof of Theorem \ref{normalepi}), to prove $p\langle \pi_1,\pi_2,\pi_2\rangle  = \pi_1$ is to prove $p \langle \pi_1,\pi_2,\pi_2\rangle \langle[1,x],[1,0]\rangle = \pi_1\langle [1,x],[1,0]\rangle$, and we have
\begin{eqnarray}
     p \langle \pi_1,\pi_2,\pi_2\rangle \langle [1,x],[1,0] \rangle &= &p\langle [1,x],[1,0],[1,0]\rangle \nonumber \\
          & = & p\gamma_1 \langle 1,\iota_1[1,0]\rangle \qquad {\rm (see \, proof \, of \, Theorem \, (\ref{pregroupoid}) )} \nonumber  \\
     &=& m\langle 1,\iota_1[1,0] \rangle \nonumber \\
           &=& [1,x] \qquad \qquad \qquad ({\rm take} \, w = 1 {\rm \, in\,  diagram \, (\ref{diagram1})}) \nonumber \\
     &=& \pi_1\langle [1,x],[1,0]\rangle.  \nonumber
           \end{eqnarray}
This proves $p\langle \pi_1,\pi_2,\pi_2\rangle = \pi_1$, and $p\langle \pi_1,\pi_1,\pi_2\rangle = \pi_2$ can be proved similarly.
\end{proof}
\begin{theorem}\label{last}
Under the assumptions of Theorem \ref{normalepi}, if $m$ is an internal multiplication $X\times Y \rightarrow A$ over $A = (A,1_A)$, then the span $\xymatrix{A/X & A \ar[l] \ar[r] & A/Y}$ admits a (unique) internal pregroupoid structure $p \colon A_3 \rightarrow A$ determined by $p\gamma_1 = m$.
\end{theorem}
\begin{proof}
The fact that $m$ determines a morphism $p \colon A_3 \rightarrow A$ with $p\gamma_1 = m$ follows from Lemma \ref{mkeriszero} and Theorem \ref{normalepi}. The fact that the so determined morphism $p \colon A_3 \rightarrow A$ is an internal pregroupoid structure on the span $$\xymatrix{A/X & A \ar[l] \ar[r] & A/Y}$$ follows from Lemma \ref{mimpliespregroupoid}.
\end{proof}
Putting together Theorems \ref{pregroupoid} and \ref{last}, we obtain:
\begin{corollary}\label{lastcoro} Under the assumptions of Theorem \ref{normalepi}, the existence of an internal multiplication $X\times Y \rightarrow A$ over $A = (A,1_A)$ is equivalent to the existence of an internal pregroupoid structure on the span
 $$\xymatrix{A/X & A \ar[l] \ar[r] & A/Y.}$$ Moreover, when such a multiplication $m$ and a pregroupoid structure $p$ exist, they (uniquely) determine each other by $p \gamma_1=m$.
\end{corollary}
\begin{remark}
 This latter result asserts in particular that the \emph{weighted centrality} of two normal monomorphisms $x \colon X \rightarrow A$ and $y \colon Y \rightarrow A$ over $1_A \colon A \rightarrow A$ is equivalent to the \emph{centrality} of the corresponding equivalence relations $A\times_{A/X} A$ and $A\times_{A/Y} A$:
$$[(X,x), (Y,y)]_{(A,1_A)} = 0 \quad \Leftrightarrow \quad [A\times_{A/X} A, A\times_{A/X} A]= \Delta_A.$$
\end{remark}
\section{Weighted subobject commutator}\label{subobject}
In the assumptions and notation of the previous section, consider the diagram
\begin{equation}\label{cuboid}
\vcenter{\xymatrix@=29pt{W \ar[rr]^{\iota_1}\ar@{=}[dd]_{}\ar[dr]^{\iota_1}&&W+Y \ar@{=}[dd]^(.7){}\ar[dr]^{[\iota_1,\iota_3]}\\&W+X \ar[rr]^{[\iota_1, \iota_2]}
\ar@{=}[dd]^{}&&W+X+Y  \ar@/^5pc/[dddd]^(.8){[w,x,y]}
%\ar`r`r[dd]`[dddd]^(.7){[x,y,z]}[dddd]
\ar[dd]_(.4){\langle [\iota_1, \iota_2, 0], [\iota_1, 0, \iota_2]\rangle }\\W \ar@{=}[dd] \ar[rr]^(.7){\iota_1}\ar[dr]^{\iota_1}&&W+Y \ar@{=}[dd] \ar[dr]^-{\langle \iota_1[1,0], 1 \rangle}\\&W+X \ar@{=}[dd]
\ar[rr]^(.6){\langle 1, \iota_1 [1,0]\rangle }&&(W+X) \times_W (W+Y) \ar@{.>}[dd]^m \\
W \ar[rr]^(.7){\iota_1} \ar[dr]_{\iota_1}&  &W+Y \ar[dr]^{[w,y]}& \\
&W+X \ar[rr]_{[w,x]}& & A&
}}
\end{equation}
in which all coproduct injections are denoted by $\iota$'s with appropriate indices, and the dotted arrow, denoted by $m$, is an arbitrary morphism with the domain and codomain as shown.
\begin{remark}\label{observation}
\begin{enumerate}
\item[(a)] The diagram obtained from (\ref{cuboid}) by removing the dotted arrow always commutes.
\item[(b)] The bottom cube in (\ref{cuboid}) commutes if and only if $m$ is an internal multiplication $X \times Y \rightarrow A$ over $W=(W,w)$.
\item[(c)] The top parallelogram in (\ref{cuboid}) is obviously a pushout.
\item[(d)] As follows from $(a)$-$(c)$, the diagram (\ref{cuboid}) commutes if and only if $m$ is an internal multiplication $X \times Y \rightarrow A$ over $W= (W,w)$.
\item[(e)] The diagram commutes whenever $m \langle [\iota_1, \iota_2,0], [\iota_1,0,\iota_2]\rangle = [w,x,y]$.
\item[(f)] As follows from $(d)$ and $(e)$, $m$ is an internal multiplication $X \times Y \rightarrow A$ over $W =(W,w)$ if and only if $m\langle [\iota_1, \iota_2,0], [\iota_1,0,\iota_2]\rangle = [w,x,y]$.
\end{enumerate}
\end{remark}
\begin{lemma}\label{canonicalepi}
If $\mathbb C$ is a regular Mal'tsev category, then the morphism \begin{equation}\label{canonical}
\langle [\iota_1, \iota_2, 0], [\iota_1, 0, \iota_2]\rangle \colon W +X+Y \rightarrow (W+X) \times_W (W+Y)
\end{equation}
is a regular epimorphism. In particular if $\mathbb C$ is a normal Mal'tsev category, then the morphism $(\ref{canonical})$ is a normal epimorphism.
\end{lemma}
\begin{proof}
Consider the category $\mathsf{Pt}(W)=((W, 1_W)\downarrow (\mathbb C \downarrow W))$. As follows from the results of \cite{Bourn96}, this category is unital. Since it is also regular, for each two objects $U$ and $V$ in it, the canonical morphism $U+V \rightarrow U \times V$ is a regular epimorphism. Now we take $U=(W+X, [1,0], \iota_1)$ and $V=(W+Y, [1,0], \iota_1)$, and the canonical morphism above becomes nothing but the morphism (\ref{canonical}). Since a morphism in $\mathsf{Pt}(W)$ is a regular epimorphism if and only if it is a regular epimorphism in $\mathbb C$, this gives the desired conclusion.
\end{proof}
\begin{definition}\label{subobjectcommutator}
Suppose $\mathbb C$ admits the (extremal epi, mono) factorization system. %(or, more generally, $\mathbb C$ is equipped with an arbitrary factorization system).
The $(W,w)$-\emph{weighted subobject commutator} $\kappa \colon [(X,x), (Y,y)]_{(W,w)} \rightarrow A$ is the image under $[w,x,y] \colon W+X+Y \rightarrow A$ of the kernel of the morphism (\ref{canonical}).
\end{definition}
When $W=0$, or $w$ is the identity morphism of $A$, we shall say ``$0$-weighted'' and write $[(X,x), (Y,y)]_0$, or say ``1-weighted'' and write $[(X,x),(Y,y)]_1$, respectively, instead of saying ``$(W,w)$-weighted'' and writing $[(X,x), (Y,y)]_{(W,w)}$.
\begin{remark}\label{Higgins}
In the case where $W=0$ the subobject commutator as defined above agrees with the categorical version of Higgins' commutator as defined by Mantovani and Metere in any ideal determined category $\mathbb C$ (see \cite{MM}, Definition $5.1$). In particular, this shows that Higgins' commutator for varieties of $\Omega$-groups \cite{Hi} is an example of $0$-weighted subobject commutator.
\end{remark}
From Remark \ref{observation}(f) and Lemma \ref{canonicalepi} we obtain:
\begin{theorem}
If $\mathbb C$ is a normal Mal'tsev category, then the following conditions are equivalent:
\begin{enumerate}
\item[(a)]  $[(X,x), (Y,y)]_{(W,w)}=0$;
\item[(b)] there exists a unique internal multiplication $X \times Y \rightarrow A$ over $W= (W,w)$.
\end{enumerate}
\end{theorem}

\begin{corollary}
Under the assumptions of Theorem \ref{normalepi}, the following conditions are equivalent:
\begin{enumerate}
\item[(a)] the span $\xymatrix{A/X & A \ar[l] \ar[r] & A/Y}$ admits an internal pregroupoid structure;
\item[(b)] the span $\xymatrix{A/X & A \ar[l] \ar[r] & A/Y}$ admits a unique internal pregroupoid structure;
\item[(c)] $[(X,x), (Y,y)]_{1}=0$.
\end{enumerate}
\end{corollary}

\section{Weighted normal commutator}\label{weightednormal}
Let $\mathbb C$ be again a pointed category with finite limits and finite colimits. We are still considering $(X,x)$, $(Y,y)$, and $(W,w)$ as above, but now we are not fixing them, but considering the category of all such triples. More precisely, by a \emph{weighted cospan} in $\mathbb C$ we shall mean a diagram in $\mathbb C$ of the form
$$\xymatrix{&W \ar[d]^w & \\X \ar[r]^x  &A & Y \ar[l]_y
}$$
that we shall denote by $(A,X,x,Y,y,W,w)$. We define morphisms of weighted cospans as the usual diagram morphisms, and we therefore form the category $\mathsf{CS}_W (\mathbb C)$ of weighted cospans. According to Definition \ref{defmultiplication}, by a \emph{multiplicative weighted cospan} in $\mathbb C$ we shall mean a pair $(C,m)$, in which $C=(A,X,x,Y,y,W,w)$ is a weighted cospan and $m$ is an internal multiplication $X \times Y\rightarrow A$ over $W=(W,w)$. Using the obvious morphisms for multiplicative weighted cospans in $\mathbb C$, we define their category $\mathsf{MCS}_W (\mathbb C)$, and the forgetful functor
\begin{equation}\label{forgetful}
\mathsf{MCS}_W (\mathbb C) \rightarrow \mathsf{CS}_W (\mathbb C).
\end{equation}
Given $C=(A,X,x,Y,y,W,w)$ in $\mathsf{CS}_W (\mathbb C)$ consider the diagram
\begin{equation}\label{pushout}
\xymatrix{W + X +Y \ar[r]^-{[w,x,y]} \ar[d]_{\langle [\iota_1, \iota_2, 0], [\iota_1, 0, \iota_2] \rangle} & A \ar[d]^{\nu_C}\\
(W+X) \times_W (W+Y) \ar[r]_-{\mu_C}& \tilde{C}
}
\end{equation}
constructed as the pushout of $\langle [\iota_1, \iota_2, 0 ], [\iota_1, 0, \iota_2]\rangle$ and $[w,x,y]$. Straightforward comparison of diagrams (\ref{cuboid}) and (\ref{pushout}) gives
\begin{theorem}\label{leftadjoint}
The functor $(\ref{forgetful})$ has a left adjoint $\mathsf{CS}_W (\mathbb C) \rightarrow \mathsf{MCS}_W (\mathbb C)$ sending $(A,X,x,Y,y,W,w)$ from $\mathsf{CS}_W (\mathbb C)$ to
$$(( \tilde{C},X,\nu_C x, Y, \nu_C y, W, \nu_C w),\mu_C) \in \mathsf{MCS}_W (\mathbb C),$$
in the notation of $(\ref{pushout})$, with the $C$-component of the unit of the adjunction given by the identity morphisms of $X$, $Y$ and $W$, and the morphism $\nu_C \colon A \rightarrow \tilde{C}$.
\end{theorem}
\begin{definition}\label{normalcommutator}
The $(W,w)$-\emph{weighted normal commutator} $$\kappa_N \colon N[(X,x),(Y,y)]_{(W,w)} \rightarrow A$$ is the kernel of the morphism $\nu_C \colon A \rightarrow \tilde{C}$ defined via the pushout $(\ref{pushout})$.
\end{definition}
\begin{remark}\label{remarksonormal}
\begin{enumerate}
\item[(a)] Looking at an object  $C=(A,X,x,Y,y,W,w)$ in $\mathsf{CS}_W (\mathbb C)$ as the corresponding object $A$ in $\mathbb C$ equipped with a structure, and looking at the objects of $\mathsf{MCS}_W (\mathbb C)$ similarly, we can identify $\nu_C \colon A \rightarrow \tilde{C}$ with the unit of the adjunction above. We can say that the $(W,w)$-weighted normal commutator $\kappa_N \colon N[(X,x),(Y,y)]_{(W,w)} \rightarrow A$ is simply the kernel of that $C$-component. This shows similarity between Definition \ref{normalcommutator} and the definition of commutator introduced in \cite{JaP}. On the other hand, as follows from the results of M.C. Pedicchio \cite{P1}, her definition given in \cite{P} could be formulated in the same way as in \cite{JaP}, but with pregroupoids instead of pseudogroupoids (used in \cite{JaP}).
\item[(b)] Defining $\tilde{C}$, $\nu_C$ and $\mu_C$ via the pushout (\ref{pushout}) is actually the same as defining them via the commutative diagram
\begin{equation}\label{colimit}
\xymatrix{ &(W+X) \times_W (W+Y)  \ar@{.>}[d]_{\mu_C} & \\
W+X \ar[dr]_{[w,x]}  \ar[ru]^-{\langle 1 , \iota_1[1,0] \rangle} \ar@{.>}[r] &\tilde{C} & W+Y  \ar[dl]^{[w,y]}  \ar@{.>}[l]  \ar[ul]_-{\langle \iota_1 [1,0], 1\rangle} \\
& A \ar@{.>}[u]^{\nu_C} & }
\end{equation}
in which the dotted arrows are required to form the colimiting cocone over the diagram formed by the solid arrows. This shows similarity between Definition \ref{normalcommutator} and the definition of commutator used in \cite{Strongly}.
\end{enumerate}
\end{remark}
Similarly to the notations introduced for the weighted subobject commutator in the previous section, when $W=0$, or $w$ is the identity morphism of $A$, we shall say ``$0$-weighted'' and write $N[(X,x), (Y,y)]_0$, or say ``$1$-weighted'' and write $N[(X,x), (Y,y)]_1$, respectively, instead of saying ``(W,w)-weighted'' and writing $N[(X,x), (Y,y)]_{(W,w)}$.

\noindent Comparing Definitions \ref{subobjectcommutator} and \ref{normalcommutator} we obtain:
\begin{theorem}\label{normalclosure}
If $\mathbb C$ admits the (extremal epi, mono) factorization system, then, for any $C= (A,X,x,Y,y,W,w)$ in $\mathsf{CS}_W (\mathbb C)$, the $(W,w)$-weighted normal commutator $\kappa_N \colon N[(X,x), (Y,y)]_{(W,w)} \rightarrow A$ is the normal closure of the $(W,w)$-weighted subobject commutator $\kappa \colon [(X,x), (Y,y)]_{(W,w)} \rightarrow A$.
\end{theorem}
\begin{corollary}\label{corollaryidealdet}
If $\mathbb C$ is ideal determined and, in the notations above, the arrow $[w,x,y] \colon W+X+Y \rightarrow A$ is a normal epimorphism, then the $(W,w)$-weighted normal commutator $\kappa_N \colon N[(X,x),(Y,y)]_{(W,w)} \rightarrow A$ agrees with the $(W,w)$-weighted subobject commutator $\kappa \colon [(X,x), (Y,y)]_{(W,w)} \rightarrow A$. ÊIn particular $$N[(X,x),(Y,y)]_{1} =  [(X,x), (Y,y)]_{1}$$ whenever $\mathbb C$ is ideal determined.
\end{corollary}
\begin{remark}
The $1$-weighted (normal) commutator has been also studied by S. Mantovani in \cite{Ma}, by adopting a different approach. In that article the term Ursini commutator is used to denote what we call here the $1$-weighted normal commutator.
\end{remark}
\section{The Huq commutator}\label{SectionHuq}
When $W=0$, the diagram (\ref{diagram1}) becomes
$$\xymatrix@=30pt{
X \ar[r]^-{\langle 1,0 \rangle} \ar[dr]_x & X \times Y \ar[d]^m & Y \ar[l]_-{\langle 0, 1 \rangle} \ar[dl]^{y} \\
&A&
}$$
and so in this case $(X,x)$ commutes with $(Y,y)$ over $(W,w)$ in the sense of Definition \ref{defmultiplication}  if and only if $(X,x)$ commutes with $(Y,y)$ in the sense of \cite{H}. Moreover, the commutator
$N[(X,x),(Y,y)]_0$ defined in the previous section becomes nothing but the commutator of $(X,x)$ and $(Y,y)$ in the sense of \cite{H}, except that the conditions required on the ground category $\mathbb C$ in \cite{H} are much stronger of course. On the other hand, say, in the case of groups the commutator $[(X,x),(Y,y)]_0$, defined in Section \ref{subobject}, becomes nothing but the classical commutator $[x(X),y(Y)]$ of the images of $X$ and $Y$ in $A$. It agrees with $N[(X,x),(Y,y)]_0$ if and only if it is a normal subgroup in $A$, which, in particular, is the case when either $x(X)$ and $y(Y)$ are normal subgroups in $A$, or their union generates $A$. We shall come back to the case of groups in Section \ref{universalcontext}.
\begin{remark}
From Theorem \ref{normalclosure} and Remark \ref{Higgins} one can deduce Proposition $5.7$ in \cite{MM}, asserting that the Huq commutator is the normalisation of the Higgins commutator: indeed, this latter result is obtained as the special case where $w \colon W \rightarrow A$ is the morphism $0 \colon 0 \rightarrow A$.
\end{remark}

\section{Preservation of commutators by normal-epimorphic images}\label{Epimorphic}

In this section, for simplicity, the ground category $\mathbb C$ is supposed to be semi-abelian \cite{JMT}.

For $X$, $Y$, and $W$ in $\mathbb C$, we shall write

$$
X \otimes^W Y = \mathsf{Ker} (\langle [\iota_1,\iota_2,0],[\iota_1,0,\iota_2] \rangle : W+X+Y \rightarrow (W+X)\times_W (W+Y)),
$$

which gives us, by Lemma \ref{canonicalepi}, a short exact sequence
\begin{equation}\label{tensorexactsequence}
 \xymatrix{ 0 \ar[r] & X\otimes^W Y \ar[r] & W+X+Y \ar[r] & (W+X)\times_W (W+Y) \ar[r] & 0}
\end{equation}
functorial in $X$, $Y$, and $W$.
\begin{lemma}\label{surjectivitytensor}
If $f \colon X ' \rightarrow X$, $g \colon Y ' \rightarrow Y$, and $h \colon W ' \rightarrow W$ are normal epimorphisms, then so is the induced morphism $f\otimes^h g : X '\otimes^{W '}Y ' \rightarrow X\otimes^W Y.$
\end{lemma}
\begin{proof}
 Consider the commutative diagram
$$
 \xymatrix@=18pt{& & 0 \ar[d] &0 \ar[d]& \\
 & &{\mathsf{Ker} (h+f+g)} \ar@{.>}[r] \ar[d]^{\mathsf{ker}(h+f+g)} &{\mathsf{Ker} (h+f) \times_{\mathsf{Ker}(h)} \mathsf{Ker} (h+g)} \ar[d]^{\mathsf{Ker} (h+f) \times_{\mathsf{Ker}(h)} \mathsf{Ker} (h+g)}& \\
 0 \ar[r] &{X' \otimes^{W'} Y' }\ar[d]^{f \otimes^h g} \ar[r] & {W'+X'+Y'} \ar[r] \ar[d]^{h+f+g}&  (W'+X')\times_{W'} (W'+Y') \ar[r] \ar[d]^{(h+f)\times_h (h+g)} &0 \\
  0 \ar[r] &X \otimes^{W } Y  \ar[r] & W+X+Y \ar[r] \ar[d] &  (W+X)\times_{W} (W+Y) \ar[d]  \ar[r] &0 \\
  & & 0& 0&
 }
 $$
 whose rows and columns are short exact sequences. It shows that $$f \otimes^h g \colon X' \otimes^{W'} Y' \rightarrow X \otimes^W Y $$ is a normal epimorphism if and only if so is the (dotted) morphism $$\xymatrix{{\mathsf{Ker} (h+f+g)} \ar@{.>}[r] & {\mathsf{Ker} (h+f) \times_{\mathsf{Ker}(h)} \mathsf{Ker} (h+g)}.}$$ Next, consider the diagram
 $$
 \xymatrix@=25pt{{\mathsf{Ker} (h+f) \times_{\mathsf{Ker}(h)} \mathsf{Ker} (h+g)} \ar@<2pt>[r]^-{\pi_2} \ar@<2pt>[d]^{\pi_1} & \mathsf{Ker}(h+g) \ar@<1pt>[l]^-{\langle sv, 1\rangle}  \ar@<2pt>[d]^{v} \ar[r]^-{\mathsf{ker}(h+g)}& W'+Y'  \ar[r]^{h+g }\ar@<2pt>[d]^{[1,0]} &W+Y \ar@<2pt>[d]^{[1,0]} \\
 \mathsf{Ker}(h+f) \ar[d]_{\mathsf{ker} (h+f)} \ar@<2pt>[r]^-{u}  \ar@<1pt>[u]^-{\langle 1, tu \rangle} & \mathsf{Ker}(h) \ar@<1pt>[u]^-t  \ar[d]^{\mathsf{ker}(h)}  \ar[r]^{\mathsf{ker}(h)} \ar@<1pt>[l]^-s & W' \ar@<1pt>[u]^{\iota_1} \ar[r]^h & W \ar@<1pt>[u]^{\iota_1} \\
 W'+X' \ar@<2pt>[r]^{[1,0]}   \ar[d]_{h+f}& W' \ar@<1pt>[l]^{\iota_1} \ar[d]^{h} & &    \\
  W +X \ar@<2pt>[r]^{[1,0]} & W \ar@<1pt>[l]^{\iota_1}&   &
 }
 $$
 in which $u$ and $s$ are induced by the suitable $[1,0]$ and $\iota_1$ respectively, and $v$ and $t$ are defined similarly. Just as in Lemma \ref{canonicalepi}, as follows from the results of \cite{Bourn96}, the morphisms $\langle 1, tu\rangle$ and $\langle sv, 1\rangle$ are jointly normal epic. Let then $k \colon \mathsf{Ker}(h+f) \rightarrow \mathsf{Ker}(h+f+g)$ and $l \colon \mathsf{Ker}(h+g) \rightarrow \mathsf{Ker}(h+f+g)$ be the arrows induced by $[\iota_1, \iota_2] \colon W' + X' \rightarrow W' + X' +Y'$ and by $[\iota_1, \iota_3] \colon W' + Y' \rightarrow W' + X' +Y'$, respectively. The result then follows from the commutativity of the diagram
 $$
 \xymatrix@=30pt{ & \mathsf{Ker}(h+f+g)  \ar@{.>}[d]& \\
  \mathsf{Ker} (h+f) \ar[r]_-{\langle 1, tu \rangle} \ar[ru]^-k &  {\mathsf{Ker} (h+f) \times_{\mathsf{Ker}(h)} \mathsf{Ker} (h+g)}  &   {\mathsf{Ker} (h+g),} \ar[l]^-{\langle sv, 1 \rangle} \ar[lu]_-l
 }
 $$
 which shows that the vertical dotted arrow is a normal epi.

% Therefore, in order to show that the morphism $\mathsf{Ker}(h+f+g) \rightarrow {\mathsf{Ker} (h+f) \times_{\mathsf{Ker}(h)} \mathsf{Ker} (h+g)}$ is a normal epimorphism, it suffices to show that both composites
 %$$\mathsf{Ker} (h+f) \rightarrow  \mathsf{Ker} (h+f) \times_{\mathsf{Ker}(h)} \mathsf{Ker} (h+g) \rightarrow  (W'+X')\times_{W'} (W'+Y')$$
 %and
 %$$\mathsf{Ker} (h+g) \rightarrow  \mathsf{Ker} (h+f) \times_{\mathsf{Ker}(h)} \mathsf{Ker} (h+g) \rightarrow  (W'+X')\times_{W'} (W'+Y') $$
 %factor through the composite
 %$$\mathsf{Ker}(h+f+g) \rightarrow W'+X'+Y' \rightarrow  (W'+X')\times_{W'} (W'+Y').$$
 %However, for the first composite this follows from the commutativity of the diagram
 %$$
 %\xymatrix{ \mathsf{Ker} (h+f)  \ar[ddrrr]^{\langle 1, tu \rangle}  \ar[dd]_{\mathsf{ker}(h+f)} \ar[ddr]_k & & \\
 %& && & \\
% W+X \ar[ddr]_{[\iota_1, \iota_2]} &\mathsf{Ker}(h+f+g)  \ar[dd]^{\mathsf{ker}(h+f+g)} & & {\mathsf{Ker} (h+f) \times_{\mathsf{Ker}(h)} \mathsf{Ker} (h+g)} \ar[dd]^{\mathsf{ker}(h+f) \times_{\ker(h)} \mathsf{ker}(h+g)}& \\
% & && \\
 % &W+X+Y \ar[rr]_-{\langle [\iota_1, \iota_2, 0], [\iota_1,0,\iota_2] \rangle} && (W+X)\times_W (W+Y)& \\ } $$
% where $k$ is induced by $[\iota_1, \iota_2]$, and for the second composite it can be shown similarly.
\end{proof}
\begin{theorem}\label{surjectivityinducedmap}
In the notation of Definitions \ref{subobjectcommutator} and \ref{normalcommutator}, for any normal epimorphisms $f \colon X' \rightarrow X$, $g \colon Y' \rightarrow Y$, $h \colon W' \rightarrow W$, $\alpha \colon A' \rightarrow A$ and morphisms $x' \colon X' \rightarrow A'$, $y' \colon Y' \rightarrow A'$ and $w' \colon W' \rightarrow A'$ in a semi-abelian category $\mathbb C$, with $xf = \alpha x'$, $yg= \alpha y'$, $wh = \alpha w'$, the induced morphisms
$$[(X',x') , (Y',y')]_{(W',w')} \rightarrow [(X,x) , (Y,y)]_{(W,w)}, $$
$$N[(X',x') , (Y',y')]_{(W',w')} \rightarrow N[(X,x) , (Y,y)]_{(W,w)} $$
also are normal epimorphisms. In particular, so are the induced morphisms
$$[(X',x') , (Y',y')]_0 \rightarrow  [(X,x) , (Y,y)]_0,$$
$$N[(X',x') , (Y',y')]_0 \rightarrow  N[(X,x) , (Y,y)]_0, $$
$$[(X',x') , (Y',y')]_1 \rightarrow  [(X,x) , (Y,y)]_1, $$
$$N[(X',x') , (Y',y')]_1 \rightarrow  N[(X,x) , (Y,y)]_1.$$
\end{theorem}
\begin{proof}
For the morphism $[(X',x') , (Y',y')]_{(W',w')}  \rightarrow [(X,x) , (Y,y)]_{(W,w)}$ this follows from Lemma \ref{surjectivitytensor}. Consequently, since normal monomorphisms have normal images under normal epimorphisms, Theorem \ref{normalclosure} insures that the same is true for the morphism $N[(X',x') , (Y',y')]_{(W',w')} \rightarrow N[(X,x) , (Y,y)]_{(W,w)}$.
\end{proof}
\begin{corollary}\label{isocommutators}
If, in addition to the assumptions of Theorem \ref{surjectivityinducedmap}, $\alpha \colon A' \rightarrow A$ is an isomorphism, then all the induced morphisms considered in Theorem \ref{surjectivityinducedmap} are isomorphisms.
\end{corollary}
\section{The Pedicchio commutator}\label{Pedicchio}
Suppose again that the ground category $\mathbb C$ is semi-abelian, and we shall have in mind the commutative diagram
\begin{equation}\label{functors}
\xymatrix{
\mathsf{Pgrd}(\mathbb C) \ar[r] \ar[d] & {\mathsf{S}(\mathbb C)} \ar[d] \\
\mathsf{MCS}_1 (\mathbb C) \ar[d] \ar[r] & \mathsf{CS}_1(\mathbb C) \ar[d] \\
\mathsf{MCS}_W (\mathbb C) \ar[r] &  \mathsf{CS}_W(\mathbb C)
}
\end{equation}
of fully faithful functors, in which:
\begin{itemize}
\item $\mathsf{Pgrd}(\mathbb C)$ and $\mathsf{S}(\mathbb C)$ denote the categories of internal pregroupoids in $\mathbb C$ and spans in $\mathbb C$, respectively;
\item $\mathsf{MCS}_1 (\mathbb C)$ and $\mathsf{CS}_1(\mathbb C)$ are the full subcategories in $\mathsf{MCS}_W (\mathbb C)$ and in $ \mathsf{CS}_W(\mathbb C)$ which are $1$-weighted, that is, which are of the form $((A,X,x,Y,y,A,1_A),m)$ and $(A,X,x,Y,y,A,1_A)$, respectively;
\item the functor ${\mathsf{S}(\mathbb C)} \rightarrow  \mathsf{CS}_1(\mathbb C)$ sends any span $S$
\begin{equation}\label{span}
\xymatrix{S_0 & S_1 \ar[l]_d \ar[r]^c& S_0'
}
\end{equation}
to the weighted cospan
\begin{equation}\label{quotientspan}
\xymatrix{& S_1 \ar@{=}[d] & \\
\mathsf{Ker}(d) \ar[r]^-{\mathsf{ker}(d)} & S_1  & {\mathsf{Ker}(c);} \ar[l]_-{\mathsf{ker}(c)}
}
\end{equation}
\item the functor  $\mathsf{Pgrd}(\mathbb C) \rightarrow \mathsf{MCS}_1 (\mathbb C)$ is induced by the functor $\mathsf{S}(\mathbb C) \rightarrow \mathsf{CS}_1 (\mathbb C)$ having in mind that:
$(a)$ to give a pregroupoid structure on the span (\ref{span}) is the same as to give a pregroupoid structure on the span $$\xymatrix{S_1/\mathsf{Ker}(d)& S_1 \ar[l] \ar[r]&S_1/\mathsf{Ker}(c)};$$
$(b)$ the pregroupoid structure on $\xymatrix{S_1/\mathsf{Ker}(d)& S_1 \ar[l] \ar[r]&S_1/\mathsf{Ker}(c)}$ \\ makes the weighted cospan (\ref{quotientspan}) multiplicative by Theorem \ref{pregroupoid};
\item the horizontal arrows in (\ref{functors}) are the forgetful functors, while the bottom vertical arrows are the (full) inclusion functors.
\end{itemize}
We can say that the weighted normal commutator theory studies the left adjoint of the functor $\mathsf{MCS}_W \rightarrow \mathsf{CS}_W(\mathbb C)$, while the Pedicchio commutator theory studies the left adjoint of the functor $\mathsf{Pgrd}(\mathbb C) \rightarrow
\mathsf{S}(\mathbb C)$. And our aim in this section is to explain the following:

\vspace{3mm}

\noindent \emph{In the semi-abelian context, the $1$-weighted normal commutator theory restricted to cospans of morphisms with normal images is equivalent to the Pedicchio commutator theory.}

\vspace{3mm}

\noindent This is not as simple as just to say that the top and the middle row in (\ref{functors}) are equivalent in any sense though. The explanation will consist of several observations below based on the previous results. A different proof from the one below has been independently found by S. Mantovani \cite{Ma}.

Here again, we assume for simplicity that the ground category $\mathbb C$ is semi-abelian.
\begin{obs}\label{ob1}
{\rm In Definition \ref{subobjectcommutator}, by Corollary \ref{isocommutators}, replacing $(X,x)$, $(Y,y)$, and $(W,w)$ with $(X ',xx')$, $(Y ',yy')$, and $(W ',ww')$, where $x' \colon X ' \rightarrow X$, $y' \colon Y ' \rightarrow Y$, and $w' \colon W ' \rightarrow W$ are regular (=normal) epimorphisms, will not change neither the commutator
$\kappa \colon [(X,x),(Y,y)]_{(W,w)} \rightarrow A$ nor the commutator \\ $\kappa_N \colon N[(X,x),(Y,y)]_{(W,w)} \rightarrow A$. This means that all properties of both commutators can be reduced to the case where $x, y$, and $w$ are monomorphisms. In particular the properties where $x$ and $y$ are required to have normal images reduce to their special cases where $x$ and $y$ are required to be normal monomorphisms. }
\end{obs}

\begin{obs}\label{obs}
{\rm According to Observation \ref{ob1}, one might wish to replace the categories $\mathsf{MCS}_W (\mathbb C)$ and   $\mathsf{CS}_W(\mathbb C)$ with their full subcategories $\mathsf{MonoMCS}_W ({\mathbb C})$ and $\mathsf{MonoCS}_W ({\mathbb C})$ with objects having $x$,$y$ and $w$ (in the notation above) monomorphisms - or even with $\mathsf{NMonoMCS}_W ({\mathbb C})$ and $\mathsf{NMonoCS}_W ({\mathbb C})$, where $x$,$y$ and $w$ are required to be normal monomorphisms. Note that:
\begin{enumerate}
\item[(a)] In the notation of Theorem \ref{leftadjoint}, the left adjoint of the forgetful functor  $\mathsf{MonoMCS}_W ({\mathbb C}) \rightarrow \mathsf{MonoCS}_W ({\mathbb C})$ will send $C=(A,X,x,Y,y,W,w)$ not to
$((\tilde{C}, X,\nu_C x, Y, \nu_C y, W,\nu_C w), \mu_C)$ but to $((\tilde{C}, X', x', Y', y', W', w'), \mu_C') $ with: the same $\tilde{C}$; $x' \colon X' \rightarrow \tilde{C}$, $y' \colon Y' \rightarrow \tilde{C}$, $w' \colon W' \rightarrow \tilde{C}$ being the images of $\nu_C x$, $\nu_C y$, $\nu_C w$, respectively; $\mu_C'$ being induced by $\mu_C$. This can be deduced from Corollary \ref{isocommutators} and the fact that $\nu_C$ is a normal epimorphism (by Lemma \ref{canonicalepi}).
\item[(b)] An advantage of using $\mathsf{MonoMCS}_W (\mathbb C)$ and  $\mathsf{MonoCS}_W (\mathbb C)$ instead of \\ $\mathsf{MCS}_W (\mathbb C)$ and  $\mathsf{CS}_W (\mathbb C)$ is that the left adjoint of the forgetful functor $\mathsf{MonoMCS}_W (\mathbb C) \rightarrow \mathsf{MonoCS}_W (\mathbb C)$ restricts to  the left adjoint of the forgetful functor $\mathsf{MonoMCS}_1 (\mathbb C) \rightarrow \mathsf{MonoCS}_1 (\mathbb C)$ (using obvious notation). It therefore presents the ``$1$-weighted'' case as a special case of the ``general-weighted'' case more naturally.
\item[(c)] Using the fact that normal monomorphisms have normal images under normal epimorphisms, we easily conclude that the left adjoint of the forgetful functor $\mathsf{MonoMCS}_W ({\mathbb C}) \rightarrow \mathsf{MonoCS}_W ({\mathbb C})$ also restricts to the left adjoint of the forgetful functor $\mathsf{NMonoMCS}_W ({\mathbb C}) \rightarrow \mathsf{NMonoCS}_W ({\mathbb C})$, and then to the left adjoint of the forgetful functor $\mathsf{NMonoMCS}_1 ({\mathbb C}) \rightarrow \mathsf{NMonoCS}_1 ({\mathbb C})$. Therefore the ``$\mathsf{NMono}$ approach'' has the same advantage as the ``$\mathsf{Mono}$ approach''.
\end{enumerate}
}
\end{obs}
\begin{obs}\label{obs2}
{\rm Let us make what we already said about the Pedicchio commutator (in Remark \ref{remarksonormal}(a) and in this section) more precise. Let $F$ be the left adjoint of the forgetful functor $\mathsf{Pgrd}(\mathbb C) \rightarrow \mathsf{S}(\mathbb C)$. As follows from the results of \cite{P1} and partly of \cite{JaP} we have:
\begin{itemize}
\item[(a)] Since the forgetful functor $\mathsf{Pgrd}(\mathbb C) \rightarrow \mathsf{S}(\mathbb C)$ is fully faithful, we can identify $\mathsf{Pgrd}(\mathbb C)$ with the full subcategory in $\mathsf{S}(\mathbb C)$ with objects all spans that admit a pregroupoid structure. After that we can say that $F$ is defined on objects by
\begin{equation}\label{reflectoronspans}
F(\xymatrix{S_0 &S_1 \ar[l]_{d}  \ar[r]^c &S_0'})= (\xymatrix{S_0 & S_1/[\mathsf{Eq}(d), \mathsf{Eq}(c)] \ar[r]^-{c^*}\ar[l]_-{d^*} & S_0')}
\end{equation}
where $[\mathsf{Eq}(d), \mathsf{Eq}(c)]$ is the Pedicchio commutator of the equivalence relations on $S_1$ which are the kernel pairs of $d$ and $c$, respectively. Conversely, it can be used to define the Pedicchio commutator via $F$. It follows from \cite{P} that this commutator $[\mathsf{Eq}(d), \mathsf{Eq}(c)]$ agrees with the Smith commutator $[\mathsf{Eq}(d), \mathsf{Eq}(c)]_{\mathsf{Smith} }$ of congruences (mentioned in the Introduction) whenever $\mathbb C$ is a Mal'tsev variety.
\item[(b)] Let us use only those spans in which $d$ and $c$ above are normal epimorphisms, and define the full subcategories $\mathsf{NEpiPgrd}(\mathbb C)$ of $\mathsf{Pgrd}(\mathbb C)$ and $\mathsf{NEpiS}(\mathbb C)$ of $\mathsf{S}(\mathbb C)$ accordingly.
Since the forgetful functor $\mathsf{Pgrd}(\mathbb C) \rightarrow \mathsf{S}(\mathbb C)$ and its left adjoint restrict to functors between $\mathsf{NEpiPgrd}(\mathbb C)$ and $\mathsf{NEpiS}(\mathbb C)$, we conclude - from the above, or directly from the results of \cite{P1} - that the left adjoint of the forgetful functor $\mathsf{NEpiPgrd}(\mathbb C) \rightarrow \mathsf{NEpiS}(\mathbb C)$ can also be defined by (\ref{reflectoronspans}).
\end{itemize}
}
\end{obs}

\begin{obs}\label{normalepi/mono}
{\rm
The functor $\mathsf{S}(\mathbb C) \rightarrow \mathsf{CS}_1(\mathbb C)$ used in (\ref{functors}) induces a category equivalence $\mathsf{NEpiS}(\mathbb C) \cong \mathsf{NMonoCS}_1 (\mathbb C),$
which is nothing but the obvious equivalence between spans of normal epimorphisms and cospans of normal monomorphisms. As follows from Corollary \ref{lastcoro}, the equivalence above induces an equivalence
$\mathsf{NEpiPgrd}(\mathbb C) \cong \mathsf{NMonoMCS}_1 (\mathbb C)$. By taking Remark \ref{remarksonormal}(a) and Observations \ref{obs} and \ref{obs2} into account we conclude:
\begin{itemize}
\item[(a)] For every two equivalence relations $E$ and $E'$ on any object $A$ in $\mathbb C$, the Pedicchio commutator $[E,E']$ is the equivalence relation $[E,E']$ on $A$ corresponding to the normal monomorphism $\kappa_N \colon N[(X,x), (Y,y)]_1 \rightarrow A$, where $x \colon X \rightarrow A$ and $y \colon Y \rightarrow A$ are the normal monomorphisms corresponding to $E$ and $E'$, respectively.
\item[(b)] For every two normal monomorphisms $x \colon X \rightarrow A$ and $y \colon Y \rightarrow A$ (with the same $A$) in $\mathbb C$, the commutator $\kappa_N \colon N[(X,x), (Y,y)]_1 \rightarrow A$ is the normal monomorphism corresponding to the Pedicchio commutator $[E,E']$, where $E$ and $E'$ are the equivalence relations on $A$ corresponding to the normal monomorphisms $x \colon X \rightarrow A$ and $y \colon Y \rightarrow A$, respectively.
\item[c)] By Corollary \ref{corollaryidealdet}, $\kappa_N \colon  N[(X,x), (Y,y)]_1 \rightarrow A$ in $(a)$ and $(b)$ is the same as $\kappa \colon [(X,x), (Y,y)]_1 \rightarrow A$. That is $(a)$ and $(b)$ equally apply to the (normal and) subobject commutator.
\end{itemize}
}
\end{obs}

\section{The universal-algebraic context}\label{universalcontext}
In this section:
\begin{itemize}
\item $\mathbb C$ denotes a pointed variety of universal algebras;
\item $A$ denotes an algebra in $\mathbb C$, and $W,X,Y$ subalgebras in $A$;
\item $w \colon W \rightarrow A$, $x \colon X \rightarrow A$, and $y \colon Y \rightarrow A$ are the inclusion maps;
\item instead of $N[(X,x), (Y,y)]_{(W,w)}$, $[(X,x), (Y,y)]_{(W,w)}$, $N[(X,x), (Y,y)]_{0}$, \\ $[(X,x), (Y,y)]_{0}$, $N[(X,x), (Y,y)]_{1}$, and $[(X,x), (Y,y)]_{1}$ we shall write \\
$N[X,Y]_{A\mid W}$, $[X,Y]_{A\mid W}$, $N[X,Y]_{A\mid 0}$, $[X,Y]_{A\mid 0}$, $N[X,Y]_{A\mid 1}$ and $[X,Y]_{A\mid 1}$;
\item if $E$ is a congruence (=equivalence relation in the sense used in the previous sections) on $A$, the corresponding normal subalgebra (=the class of $0$ under $E$) of $A$ will be denoted by $E(0)$;
\item if $M$ is a normal subalgebra of $A$, the corresponding congruence on $A$, which is the
smallest congruence $E$ on $A$ with $E(0) = M$, is $A \times_{A/M}A$;
\item $N[X,Y]_{A\mid W}$ and $[X,Y]_{A\mid W}$ (and their special cases above) will always be supposed to be subalgebras of $A$, and $\kappa_N \colon N [X,Y]_{A \mid W} \rightarrow A$ and $\kappa \colon [X,Y]_{A \mid W} \rightarrow A$ will be supposed to be inclusions maps.
\end{itemize}
\begin{obs}\label{variety}
{\rm \begin{itemize}
\item[(a)] Being a variety, $\mathbb C$ automatically has all small limits and colimits, and in particular the finite ones. Furthermore, $\mathbb C$ is Barr exact, and in particular regular.
\item[(b)] As explained in \cite{JMTU}, $\mathbb C$ is ideal determined if and only if it is (pointed) $BIT$ (``buona teoria degli ideali'') in the sense of \cite{Urs72}, or, equivalently, it is a (pointed) ideal determined variety in the sense of \cite{GU}.
\item[(c)] As explained in \cite{JMU}, $\mathbb C$ is semi-abelian if and only if it is (pointed) $BIT$ \emph{speciale} in the sense of \cite{Urs73}, or, equivalently, it is a (pointed) classically ideal determined variety in the sense of \cite{GU}.
\end{itemize}}
\end{obs}
From this observation we obtain:
\begin{corollary}\label{coronormalclosure}
{\rm
\begin{itemize}
\item[(a)] By Theorem \ref{normalclosure}, $N[X,Y]_{A \mid W}$ is the normal closure of $[X,Y]_{A \mid W}$, and in particular $N[X,Y]_{A \mid 0}$ is the normal closure of $[X,Y]_{A \mid 0}$, and $N[X,Y]_{A \mid 1}$ is the normal closure of $[X,Y]_{A \mid 1}$.
\item[(b)] By Corollary \ref{corollaryidealdet}, if $\mathbb C$ is ideal determined and the union $W \cup X \cup Y$ generates $A$, we have $N[X,Y]_{A \mid W}= [X,Y]_{A \mid W}$. In particular, if $\mathbb C$ is ideal determined, then we always have $N[X,Y]_{A \mid 1} = [X,Y]_{A \mid 1}$ and, if the union $X \cup Y$ generates $A$, then $N[X,Y]_{A \mid 0} = [X,Y]_{A \mid 0}$.
\item[(c)] By Theorem \ref{surjectivityinducedmap}, if $\mathbb C$ is semi-abelian (=classically ideal determined) and $\alpha \colon A \rightarrow A'$ is a surjective homomorphism, then
\begin{eqnarray}
\alpha ([X,Y]_{A \mid W}) &=& [\alpha (X), \alpha (Y) ]_{A' \mid \alpha(W)}, \nonumber \\ \alpha (N[X,Y]_{A \mid W}) &=& N[\alpha (X), \alpha (Y) ]_{A' \mid \alpha(W)}, \nonumber \\
\alpha ([X,Y]_{A \mid 0}) &=& [\alpha (X), \alpha (Y) ]_{A' \mid 0}, \nonumber \\
 \alpha (N[X,Y]_{A \mid 0}) &=& N[\alpha (X), \alpha (Y) ]_{A' \mid 0}, \nonumber \\
\alpha ([X,Y]_{A \mid 1}) &=& [\alpha (X), \alpha (Y) ]_{A' \mid 1}, \nonumber \\
 \alpha (N[X,Y]_{A \mid 1}) &=& N[\alpha (X), \alpha (Y) ]_{A' \mid 1}. \nonumber
 \end{eqnarray}
\item[(d)] By Observation \ref{normalepi/mono}, and since the Pedicchio commutator is a categorical generalization of the Smith commutator \cite{S}, if $\mathbb C$ is semi-abelian the commutator $[X,Y]_{1}$ generalizes the Smith commutator in the sense that
$$[E,E']_{\mathsf{Smith}}= A \times_{A/[E(0), E'(0)]_{A \mid 1}} A$$ or, equivalently, $$ [E,E']_{\mathsf{Smith}}(0) = [E(0), E'(0)]_{A \mid 1},$$
for every two congruences $E$ and $E'$ on $A$, where $[E,E']_{\mathsf{Smith}}$ denotes the Smith commutator of $E$ and $E'$. Still equivalently, if ($\mathbb C$ is semi-abelian and) $X$ and $Y$ are normal subalgebras in $A$, then
$$[X,Y]_{A \mid 1}=  [A \times_{A/X} A, A \times_{A/Y} A]_{\mathsf{Smith}}(0),$$
or, equivalently,$$A \times_{A/[X,Y]_{A \mid 1}} A = [A \times_{A/X}A, A \times_{A/Y}A]_{\mathsf{Smith}}.$$
\end{itemize}
}
\end{corollary}
Another notion of commutator was introduced in \cite{Urs94}. It can be defined as follows:
\begin{definition}\label{commutatorterms}
 (a) A term $t({\bf{w}}_1, \dots, {\bf{w}}_k, {\bf{x}}_1,\dots,  {\bf{x}}_m, {\bf{y}}_1, \dots , {\bf{y}}_n)$ in $\mathbb C$, in which $\{{\bf{x}}_1, \dots , {\bf{x}}_m \}$ and $\{{\bf{y}}_1, \dots , {\bf{y}}_n \}$ are disjoint sets, is said to be a commutator term in $({\bf{x}}_1, \dots , {\bf{x}}_m)$ and $({\bf{y}}_1, \dots , {\bf{y}}_n)$, if it is an ideal term in $({\bf{x}}_1,\dots,  {\bf{x}}_m)$ and in $({\bf{y}}_1, \dots , {\bf{y}}_n)$ at the same time, that is, if
$$t({\bf{w}}_1, \dots, {\bf{w}}_k, 0,\dots, 0, {\bf{y}}_1, \dots , {\bf{y}}_n) = 0 = t({\bf{w}}_1, \dots, {\bf{w}}_k, {\bf{x}}_1,\dots,  {\bf{x}}_m, 0, \dots,0).$$
The collection of all such commutator terms will be denoted by $$\mathsf{CT}_{\mathbb C} (({\bf{x}}_1, \cdots, {\bf{x}}_m),({\bf{y}}_1, \cdots, {\bf{y}}_n)).$$
(b) The commutator $[X,Y]_{{\mathbb C}, A}$ is defined as
\begin{eqnarray}
& &[X,Y]_{{\mathbb C}, A} = \{ t_A (w_1, \dots , w_k, x_1, \dots , x_m, y_1, \dots , y_n) \in A \, \mathrm{\, such \, that :}  \qquad \qquad  \nonumber \\
& & w_1, \dots, w_k \in A; x_1, \dots, x_m \in X; y_1, \dots, y_n \in Y; \nonumber \\
 & & t({\bf{w}}_1, \dots, {\bf{w}}_k, {\bf{x}}_1,\dots,  {\bf{x}}_m, {\bf{y}}_1, \dots , {\bf{y}}_n)   \in \mathsf{CT}_{\mathbb C} (({\bf{x}}_1, \cdots, {\bf{x}}_m),({\bf{y}}_1, \cdots, {\bf{y}}_n)) \},\nonumber
\end{eqnarray}
where $t_A (w_1, \dots , w_k, x_1, \dots , x_m, y_1, \dots , y_n)$ denotes $$t({\bf{w}}_1, \dots, {\bf{w}}_k, {\bf{x}}_1,\dots,  {\bf{x}}_m, {\bf{y}}_1, \dots , {\bf{y}}_n)$$ calculated in $A$ when ${\bf{w}}_1, \dots, {\bf{w}}_k, {\bf{x}}_1,\dots, { \bf{x}}_m, {\bf{y}}_1, \dots , {\bf{y}}_n$ are substituted by\\ $w_1, \dots, w_k, x_1, \cdots, x_m, y_1, \dots, y_n$.
\end{definition}

\begin{remark}
In fact $[X,Y]_{\mathbb C, A}$ is defined in \cite{Urs94} more generally, when $X$ and $Y$ are arbitrary subsets in $A$, not necessarily subalgebras. But it is shown there that $[X,Y]_{\mathbb C, A}$ is always an ideal in $A$, and that it will not be changed if $X$ and $Y$ are replaced by the ideals they generate.
\end{remark}
The $(b)$ part of the following proposition is similar to Proposition $3.3$ in \cite{JMU}:
\begin{proposition}\label{equalitycommutator}
\begin{itemize}
\item[(a)]  $[X,Y]_{\mathbb C, A} \subseteq [X,Y]_{ A \mid 1} $.
\item[(b)] If $A$ is the free algebra in $\mathbb C$ on a set $S$, and $X$ and $Y$ are freely generated by $P$ and $Q$ respectively, which are disjoint subsets of $S$, then
$$[X,Y]_{\mathbb C, A} = [X,Y]_{ A \mid 1}.$$
\end{itemize}
\end{proposition}
\begin{proof}
$(a)$ We have to show that, in the notation of Definition \ref{commutatorterms} (b), we have
$$t_A (w_1, \dots , w_k, x_1, \dots , x_m, y_1, \dots , y_n) \in [X,Y]_{A \mid 1} $$
whenever: $w_1, \dots, w_k$ are in $A$; $x_1, \dots, x_m$ are in $X$; $y_1, \dots, y_n$ are in $Y$; and \\Ê$t = t({\bf{w}}_1, \dots, {\bf{w}}_k, {\bf{x}}_1,\dots,  {\bf{x}}_m, {\bf{y}}_1, \dots , {\bf{y}}_n)$ is in $
\mathsf{CT}_{\mathbb C} (({\bf{x}}_1, \cdots, {\bf{x}}_m),({\bf{y}}_1, \cdots, {\bf{y}}_n)).$\\
Without loss of generality we can assume that not only the sets $\{ {\bf{x}}_1,\dots,  {\bf{x}}_m\}$ and $\{ {\bf{y}}_1, \dots , {\bf{y}}_n \}$ are disjoint, but also each of them is disjoint with $\{ {\bf{w}}_1, \dots, {\bf{w}}_k \}$. After that we consider the coproduct diagram
$$
\xymatrix{& A \ar[d]^{\iota_1} & \\
X \ar[r]^-{\iota_2}&A +X+Y  & {Y,} \ar[l]_-{\iota_3}
}
$$
and the element $$c= t_{A+X+Y} (\iota_1(w_1), \dots, \iota_1(w_k), \iota_2 (x_1), \dots, \iota_2 (x_m),\iota_3 (y_1), \dots, \iota_3 (y_n))$$
in it. Since $t$ is in $\mathsf{CT}_{\mathbb C} (({\bf{x}}_1, \cdots, {\bf{x}}_m),({\bf{y}}_1, \cdots, {\bf{y}}_n))$, $c$ belongs to the kernel $X \otimes^A Y$ of the map (\ref{canonical}) with $W=A$ and, on the other hand, it is sent to
$t_A (w_1, \dots, w_k, x_1, \dots, x_m, y_1, \dots, y_n)$ by the canonical map $A+X+Y \rightarrow A$. Therefore $t_A (w_1, \dots, w_k, x_1, \dots, x_m, y_1, \dots, y_n)$ is in $[X,Y]_{A \mid 1}$.

\noindent $(b)$: We have to show that $[X,Y]_{A \mid 1} \subseteq [X,Y]_{\mathbb C,  A}$. Any element $a$ of $[X,Y]_{A \mid 1}$ can be presented in the form
$$a = t_A (w_1, \dots w_k, x_1, \dots, x_m, y_1, \dots, y_n),$$
with $t_{A+X+Y} (\iota_1 (w_1), \dots, \iota_{1}(w_k),\iota_2 (x_1), \dots, \iota_{2}(x_m), \iota_3(y_1), \dots, \iota_3(y_n))$ in \\$X \otimes^A Y$. Since $A$ is free on $S$ and $X$ and $Y$ are free on disjoint subsets of $S$, using $S$ as our alphabet, we can assume that:
\begin{itemize}
\item the elements $w_1, \dots, w_k, x_1, \dots, x_m, y_1, \dots y_n$ are themselves terms, and moreover:
\item $w_1, \dots, w_k$ are, say, terms of variables ${\bf w}_1, \dots, {\bf w}_{k'} \in S$; $x_1, \dots, x_m$ are terms of variables ${\bf x}_1, \dots, {\bf x}_{m'} \in S$; and $y_1, \dots, y_n$ are terms of variables
${\bf y}_1, \dots, {\bf y}_{n'} \in S$ - with the sets $\{ {\bf x}_1, \dots, {\bf x}_{m'} \}$ and $\{ {\bf y}_1, \dots, {\bf y}_{n'} \}$ being disjoint.
\end{itemize}
Now let $u= u({\bf w}_1, \dots, {\bf w}_{k'}, {\bf x}_1, \dots, {\bf x}_{m'}, {\bf y}_1, \dots, {\bf y}_{n'})$ be the term $t$ considered as a term of variables ${\bf w}_1, \dots, {\bf w}_{k'}, {\bf x}_1, \dots, {\bf x}_{m'}, {\bf y}_1, \dots, {\bf y}_{n'}$. Then $u$ belongs to $\mathbf{CT}_{\mathbb C} (({\bf x}_1, \dots, {\bf x}_{m'}), ({\bf y}_1, \cdots, {\bf y}_{n'}))$ and at the same time $u=a$ as elements of $A$. Therefore $a$ belongs to $[X,Y]_{\mathbb C, A}$, as desired.
\end{proof}
\begin{theorem}\label{equalitysemiabelian}
If $\mathbb C$ is semi-abelian (=classically ideal determined), then the equality $$ [X,Y]_{\mathbb C, A} = [X,Y]_{A \mid 1}$$
always holds.
\end{theorem}

\noindent \emph{Proof 1.} Using Corollary \ref{coronormalclosure} (c) and a similar result for the commutator $[-,- ]_{\mathbb C}$, which is obvious, as mentioned in \cite{Urs94}, it suffices to show that there exists
\begin{itemize}
\item a set $S$ with disjoint subsets $P$ and $Q$, such that
\item there is a surjective homomorphism $\alpha$ from the free algebra on $S$ to $A$, under which the images of $P$ and $Q$ generate $X$ and $Y$ respectively
\end{itemize}
- which is a triviality.
\vspace{3mm}

\noindent \emph{ Proof 2.} $[X,Y]_{\mathbb C, A} = [A \times_{A/X} A, A \times_{A/Y} A](0)$ by Theorem $2.6$ \cite{GU}, and $$[A \times_{A/X} A, A \times_{A/Y} A](0)= [X,Y]_{A \mid 1}$$ by \ref{coronormalclosure}(d).
\hfill $\Box$
\begin{remark}
(a) Just as $[X,Y]_{X \mid W}$ generalizes $[X,Y]_{X \mid 1}$, we could introduce $[X,Y]_{\mathbb C, A \mid W} $ by requiring $w_1, \dots, w_k$ in Definition \ref{commutatorterms} (b) to belong to $W$. This generalization of the commutator introduced in \cite{Urs94} will obviously allow extending Proposition \ref{equalitycommutator} and Theorem \ref{equalitysemiabelian} to the case of arbitrary $W$, except that our second proof of Theorem \ref{equalitysemiabelian} would also require introducing a suitable notion of \emph{weighted span commutator}.

\noindent (b) When $\mathbb C$ is ``nice'', and $X$ and $Y$ are normal subalgebras in $A$, one expects that the Pedicchio commutator $[A \times_{A/X} A, A \times_{A/Y} A]$ agrees with the Huq commutator $[X,Y]$, that is
\begin{equation}\label{Huq=Smith}
[X,Y] = [A \times_{A/X} A, A \times_{A/Y} A](0),
\end{equation}
 or, equivalently,
 $$ A \times_{A/[X,Y]}A = [A \times_{A/X} A, A \times_{A/Y} A].$$
 The first such result, where ``nice'' was interpreted as ``strongly protomodular'' was proved in \cite{Strongly}, although its main ingredient is in fact Theorem $6.1$ in \cite{BG}. In particular this applies to all varieties of groups, of (associative, non-unital) rings, of Lie algebras over commutative rings, of (pre)crossed modules over those, and to some other classical algebraic structures.  For some special types of pairs of normal subalgebras $X$ and $Y$ it was later shown that the Pedicchio and the Huq commutators do agree even in general semi-abelian categories (see, for instance, Proposition $2.2$ in \cite{GVdL} and Proposition $4.6$ in \cite{EVdL}), although in general
 this is not the case (see the counter-example due to the second author explained in Section $6$ of \cite{Strongly}, although in fact [12] provides an earlier counter-example, as follows
from what we say at the end of $(c)$ below).

The equality (\ref{Huq=Smith}) implies
 \begin{eqnarray}[X,Y]&=&[X,Y]_{A \mid 0} = [X,Y]_{A \mid W} = [X,Y]_{A \mid 1} \nonumber \\
 & =& N[X,Y]_{A \mid 0} = N[X,Y]_{A \mid W} = N[X,Y]_{A \mid 1},\nonumber \end{eqnarray}
 for every $W$.

 However, although we always trivially have
 $$
 [X,Y]=[X,Y]_{A \mid 0} \le [X,Y]_{A \mid W} \le [X,Y]_{A \mid 1}
 $$
 and
 $$
 N[X,Y]_{A \mid 0} \le N[X,Y]_{A \mid W} \le N[X,Y]_{A \mid 1},
 $$
 when $X$ and $Y$ are not normal, we might have
 $$[X,Y]_{A \mid 0} \not= [X,Y]_{A \mid W} \not= [X,Y]_{A \mid 1},$$
 even in the varieties of groups. The same is true for commutative (associative, non-unital) rings, where all commutators can be calculated very easily, as our following Example \ref{rings} will show.
\item[(c)]  The indices $\mathbb C$ and $A$ in the symbol $[X,Y]_{\mathbb C, A}$ indicate of course that the
commutator  $[X,Y]_{\mathbb C, A}$ depends, in general, not only on $X$ and $Y$, but also on $\mathbb C$ and
$A$. However, it might be independent of $\mathbb C$ and $A$ in some sense. The concept of
independence of $\mathbb C$ is especially strange from the categorical viewpoint, but from
the universal algebra viewpoint it has a standard meaning, namely: if $\mathbb A$ is a
subvariety in $\mathbb C$ containing $A$, then $[X,Y]_{{\mathbb A}, A}= [X,Y]_{{\mathbb C}, A}$. Such independence is
actually expected under some additional conditions on $\mathbb C$, and, for instance,
subtractivity (which is much weaker than semi-abelian-ness) of $\mathbb A$ suffices, as
shown in \cite{Urs94}. The independence of $A$ is defined by saying that if $B$ is a
subalgebra of $A$, containing $X$ and $Y$, then $[X,Y]_{{\mathbb C},B} = [X,Y]_{{\mathbb C},A}$. It fails even in a
general semi-abelian variety, as the Exercise $10$ of Chapter $5$ in \cite{FM} shows, but it
is essentially equivalent to (\ref{Huq=Smith}), and so it holds in ÒniceÓ cases (see $(b)$ above).
\end{remark}
\begin{example}\label{rings}
Let $\mathbb C$ be the variety of commutative (associative, non-unital) rings. Let us recall and use the following conventions:
\begin{itemize}
\item For arbitrary $R$ and $S$ in $\mathbb C$, their classical tensor product $R\otimes S$ (over the ring $\mathbb Z$ of integers) is a commutative ring, whose multiplication is compatible with the (non-unital) $R$-module and $S$-module structure. Using this structure, and not making distinction between the left and right module structures we can present the coproduct $R+S$ as follows: as an abelian group it is the product $R \times S \times (R \otimes S)$, and its multiplication is defined by
$$(r,s,t)(r',s',t') = (rr', ss', r\otimes s' + s \otimes r' +rt'+tr'+st'+ts' +tt').$$
\item When $R$ and $S$ are subrings of $A$, we shall write $RS$ for the subring in $A$ generated by the elements in $A$ of the form $rs$ with $r$ in $R$ and $s$ in $S$, which is nothing but the set of (finite) sums of elements in $A$ of that form. We shall also write
$$R+^A S = \{ r+s + t \in A \mid r \in R,s \in S  {\, \rm and \, }  t \in RS \} ,$$
having in mind that this is the subring in $A$ generated by the union $R \cup S$, and therefore a quotient ring of the coproduct $R+S$.
\end{itemize}
Then, using the description of the coproducts in $\mathbb C$, an easy calculation gives
$$[X,Y] _{A \mid W} = XY +^A WXY,$$
from which we obtain
$$[X,Y] = [X,Y]_{A \mid 0} = XY,$$
$$[X,Y]_{A \mid 1} = N[X,Y]_{A \mid 0}= N [X,Y]_{A\mid W} = N [X,Y]_{A \mid 1}= XY +^A AXY.$$
In particular, if $X$ and $Y$ are ideals in $A$, then all these commutators coincide with $XY$.
\end{example}

\end{document}